\documentclass[twoside]{amsart}
\usepackage{latexsym}
\usepackage{amssymb,amsmath,amsopn}
\usepackage[dvips]{graphicx}   %To insert ps figures
\usepackage{color,epsfig}      %To insert ps figures 

\newtheorem{thm}{Theorem}
\newtheorem{lem}{Lemma}
\newtheorem{cor}{Corollary}
\theoremstyle{definition}
\newtheorem{defn}{Definition}

\renewcommand{\Re}{\mathbb R}

\newcommand{\M}{\mathbb{M}}

\DeclareMathOperator{\conv}{conv}

\DeclareMathOperator{\spn}{span}
\DeclareMathOperator{\id}{id}

\begin{document}
\title[On adjoint abelian operators]{On diagonalizable operators in
Minkowski spaces with the Lipschitz property}

\author[Z. L\'angi]{Zsolt L\'angi}
                                                                               
\address{Zsolt L\'angi, Dept. of Geometry, Budapest University of Technology,
Budapest, Egry J\'ozsef u. 1., Hungary, 1111}
\email{zlangi@math.bme.hu}
                                                                        

\begin{abstract}
A real semi-inner-product space is a real vector space $\M$ equipped with
a function $[.,.] : \M \times \M \to \Re$ which is linear in its first variable, strictly positive
and satisfies the Schwartz inequality.
It is well-known that the function $||x|| = \sqrt{[x,x]}$ defines a norm on $\M$. and vica versa,
for every norm on $X$ there is a semi-inner-product satisfying this equality.
A linear operator $A$ on $\M$ is called \emph{adjoint abelian with respect to $[.,.]$},
if it satisfies $[Ax,y]=[x,Ay]$ for every $x,y \in \M$.
The aim of this paper is to characterize the diagonalizable adjoint abelian operators in finite
dimensional real semi-inner-product spaces satisfying a certain smoothness condition.
\end{abstract}
\maketitle

\section{Introduction and preliminaries}

A real semi-inner-product space is a real linear space $\M$ equipped with
a function $[.,.] : \M \times \M \to \Re$, called a \emph{semi-inner-product}, such that
\begin{enumerate}
\item[(1)] $[.,.]$ is linear in the first variable,
\item[(2)] $[x,x] \geq 0$ for every $x \in \M$, and $[x,x]=0$ yields that $x=0$,
\item[(3)] $[x,y]^2 \leq [x,x]\cdot [y,y]$ for every $x,y \in \M$.
\end{enumerate}

These spaces were introduced in 1961 by Lumer \cite{L61}, and have been extensively studied since then
(cf., for example \cite{D04}).
It was remarked in \cite{L61} that in a real semi-inner-product space $\M$, the function
$||x||=\sqrt{[x,x]}$ defines a norm. The converse also holds, i.e. if $\M$ is a real linear space, then
for every real norm $||.|| : \M \to \Re$,
there is a semi-inner-product $[.,.] : \M \times \M \to \Re$ satisfying $||x||=\sqrt{[x,x]}$.
Furthermore, the semi-inner-product determined by a norm is unique if, and only if,
its unit ball is smooth; that is, if the unit sphere has a unique supporting hyperplane at its every point.
By \cite{G67}, in this case the semi-inner-product is homogeneous in the second variable;
i.e., $[x,\lambda y]= \lambda [x,y]$ for any $x,y \in \M$ and $\lambda \in \Re$.

We say that a real semi-inner-product is \emph{continuous}, if for every $x,y,z \in \M$ with
$[x,x]=[y,y]=[z,z]=1$, $\lambda \to 0$ yields that $[x,y+\lambda z] \to [x,y]$ (cf. \cite{G67} or \cite{GH}).
It is well-known that the semi-inner-product determined by a smooth norm is continuous;
it follows, for example, from $E^*$ on page 118 of \cite{RV73} and Theorem 3 of \cite{G67}.

A linear operator $A$ is called \emph{adjoint abelian with respect to a semi-inner-product $[.,.]$}, if
it satisfies $[Ax,y]=[x,Ay]$ for every $x,y \in \M$ (cf., for instance \cite{FJ74} and \cite{FJ76}).

In the following, $\M$ denotes a smooth Minkowski space; that is, a real finite dimensional smooth normed space,
and $||.||$ and $[.,.]$ denote the norm and the induced semi-inner-product of $\M$, respectively.
We denote by $S$ the unit sphere with respect to the norm, i.e., we set $S = \{ x \in \M : ||x|| = 1 \}$.
We say that the semi-inner-product $[.,.]$ has the \emph{Lipschitz property}, if for every $x \in S$,
there is a real number $\kappa$ such that for every $y,z \in S$, we have
$\left| [x,y] - [x,z] \right| \leq \kappa ||y-z||$.
We note that in a similar way, a differentiability property of semi-inner-products
was defined in \cite{GH}, and that any semi-inner-product satisfying that
differentiability property satisfies also the Lipschitz property.

The aim of this paper is to characterize the diagonalizable adjoint abelian operators
in finite dimensional spaces with a semi-inner-product that satisfies the Lipschitz property.

To formulate our main result, we need the following notions and notations.
An isometry of $\M$ is an operator $A : \M \to \M$ satisfying $|| A x|| = ||x||$ for every $x \in \M$,
or, equivalently, $[Ax,Ay]=[x,y]$ for every $x,y \in \M$ (cf. \cite{K71}).
For the properties of isometries in Minkowski spaces, the interested
reader is referred to \cite{MS09}.

For the following definition, see also \cite{G67}.

\begin{defn}
If $x,y \in \M$ and $[x,y]=0$, we say that \emph{$x$ is transversal to $y$}, or \emph{$y$ is normal to $x$}.
If $X,Y \subset \M$ such that $[x,y]=0$ for every $x \in X$ and $y \in Y$,
we say that \emph{$X$ is transversal to $Y$} or \emph{$Y$ is normal to $X$}.
\end{defn}

\begin{defn}\label{defn:directsum}
Let $U$ and $V$ be linear subspaces of $\M$ such that $\M = U \oplus V$.
If for every $x_u, y_u \in U$ and $x_v, y_v \in V$, we have
\[
[ x_u + x_v, y_u + y_v ] = [x_u,y_u] + [x_v, y_v],
\]
then we say that the semi-inner-product $[.,.]$ is the \emph{direct sum of
$[.,.]|_U$ and $[.,.]|_V$}, and denote it by $[.,.]=[.,.]|_U + [.,.]|_V$.
If there are no such linear subspaces of $\M$, we say that $[.,.]$ is \emph{non-decomposable}.
\end{defn}

We remark that if $[.,.] = [.,.]|_U + [.,.]_V$ for some linear subspaces $U$ and $V$, then $U$ and $V$
are both transversal and normal, and that the converse does not hold. We note also that any two semi-inner-product spaces can be added in this way (cf. \cite{GH}). 
Definition~\ref{defn:directsum} can be formulated for finitely many subspaces as well in the natural way.
For the simplicity of notation, we mean that every semi-inner-product space is the direct sum of itself.

Let $A : \M \to \M$ be a linear operator, and let $\lambda_1 > \lambda_2 > \ldots > \lambda_k \geq 0$
be the absolute values of the eigenvalues of $A$.
If $\lambda_i$ is an eigenvalue of $A$, then $E_i$ denotes the eigenspace of $A$ belonging to $\lambda_i$,
and if $\lambda_i$ is not an eigenvalue, we set $E_i= \{ 0 \}$.
We define $E_{-i}$ similarly with $-\lambda_i$ in place of $\lambda_i$, and set $\bar{E}_i=\spn(E_i \cup E_{-i})$.
Our main theorem is the following.

\begin{thm}\label{thm:main}
Let $\M$ be a smooth Minkowski space such that the induced semi-inner-product $[.,.] : \M \times \M \to \Re$
satisfies the Lipschitz condition, and let $A : \M \to \M$ be a diagonalizable linear operator.
Then $A$ is adjoint abelian with respect to $[.,.]$ if, and only if, the following hold.
\begin{enumerate}
\item $[.,.]$ is the direct sum of its restrictions to the subspaces $\bar{E}_i$, $i=1,2,\ldots,k$;
\item for every value of $i$, the subspaces $E_i$ and $E_{-i}$ are both transversal and normal;
\item for every value of $i$, the restriction of $A$ to $\bar{E}_i$ is the product of $\lambda_i$ and
an isometry of $\bar{E}_i$.
\end{enumerate}
\end{thm}

From Theorem~\ref{thm:main}, we readily obtain the following corollary.

\begin{cor}\label{cor:nondecomp}
Let $\M$ be a smooth Minkowski space such that the induced semi-inner-product $[.,.]$ satisfies the Lipschitz condition.
Then the following are equivalent.
\begin{enumerate}
\item $[.,.]$ is non-decomposable;
\item every diagonalizable adjoint abelian linear operator of $\M$ is a scalar multiple of an isometry of $\M$.
\end{enumerate}
\end{cor}

Note that if $A$ is not diagonal, then we may apply Theorem~\ref{thm:main} for the span of the eigenspaces of $A$.

\begin{cor}\label{cor:nondiagonal}
Let $\M$ be a smooth Minkowski space such that the induced semi-inner-product $[.,.] : \M \times \M \to \Re$
satisfies the Lipschitz condition, and
let $A : \M \to \M$ be an adjoint abelian linear operator with respect to $[.,.]$.
Then (1), (2) and (3) in Theorem~\ref{thm:main} hold for $A$ .
\end{cor}

If $[.,.] = [.,.]|_U + [.,.]|_V$ for some subspaces $U$ and $V$,
and $u \in U$ and $v \in V$, then, by Theorem~\ref{thm:main}, $S \cap \spn \{ u,v \}$ is an ellipse.
This observation is proved, for example, in Statement 1 of \cite{GH}.
Thus, we have the following.

\begin{cor}\label{cor:no_ellipse}
Let $\M$ be a smooth Minkowski space such that the induced semi-inner-product satisfies the Lipschitz condition.
If no section of the unit sphere $S$ with a plane is an ellipse with the origin as its centre,
then every diagonalizable adjoint abelian operator of $\M$ is a scalar multiple of an isometry of $\M$.
\end{cor}

In the proof of Theorem~\ref{thm:main}, we need the following lemma.

\begin{lem}\label{lem:uniform}
Let $\M$ be a smooth Minkowski space.
Let $||.||$, $[.,.]$  and $S$ denote the norm, the associated semi-inner-product and the unit sphere of $\M$.
Then the following are equivalent.
\begin{enumerate}
\item[(1)] $[.,.]$ satisfies the Lipschitz condition;
\item[(2)] for every $x \in \M$, the function $f_x : \M \to \Re$, $f_x(y) = [x,y]$ is uniformly continuous
on $\M$; that is, for every $x \in \M$ and $\varepsilon > 0$ there is a $\delta > 0$ such that $y,z \in \M$
and $||y-z|| < \delta$ imply $| [x,y] - [x,z] | < \varepsilon$;
\item[(3)] for every $x \in \M$ and any sequences $\{ y_n\}$, $\{ z_n \}$ in $\M$, if $||y_n - z_n || \to 0$,
then $| [x,y_n] - [x,z_n] | \to 0$.
\end{enumerate}
\end{lem}

\begin{proof}
Note that (2) and (3) are equivalent. We prove that (1) and (3) are equivalent.

First we show that (1) yields (3).
Observe that since $[x,y]$ is homogeneous in $x$, it suffices to prove (3) for $x \in S$.
Let $x \in S$, and assume that there is a number $\kappa \in \Re$ such that for every $y,z \in S$,
we have $\left| [x,y]-[x,z] \right| < \kappa ||y-z||$.
Consider the sequences $\{ y_n \}$ and $\{ z_n \}$ in $\M$, and assume that $||y_n - z_n || \to 0$.
Since a continuous function is uniformly continuous on any compact set and since the unit ball of $\M$
is compact, we may assume that $||y_n|| \geq 1$ and that $||z_n|| \geq 1$ for every $n$.
Let $w_n = \frac{||z_n||}{||y_n||}y_n$.
Observe that, by the definition of semi-inner-product, $\left| [u,v] \right| \leq 1$ for any $u,v \in S$.
Then, from $\left| \left[ x, \frac{y_n}{||y_n||} \right] \right| \leq 1$ and from the triangle inequality,
we obtain that
\[
\left| [x,y_n] - [x,z_n] \right| \leq \left| [x,y_n]-[x,w_n] \right| + \left| [x,w_n]-[x,z_n] \right| =
\big| ||z_n|| - ||y_n|| \big| \cdot
\]
\[
\cdot \left| \left[x, \frac{y_n}{||y_n||}\right]\right| + ||z_n|| \cdot \left| \left[x,\frac{w_n}{||w_n||}\right]-\left[x,\frac{z_n}{||z_n||}\right]
\right| \leq ||y_n - z_n|| + \kappa || w_n - z_n||.
\]
Note that $||w_n - y_n || = \big| ||y_n|| - ||z_n|| \big| \leq ||y_n - z_n|| \to 0$ and
that $||w_n-z_n|| \leq ||w_n - y_n|| + ||y_n - z_n|| \to 0$, from
which it follows that $\left| [x,y_n] - [x,z_n] \right| \to 0$.

Assume that (1) does not hold. Then there is a point $x \in S$ and sequences $y_n , z_n \in S$
such that $\left| [x,y_n] - [x,z_n] \right| = \kappa_n ||y_n - z_n ||$ where $\kappa_n \to \infty$.
We may assume that $\kappa_n > 0$ for every $n$, and since $S$ is compact,
also that $y_n \to y$ and $z_n \to z$ for some $y,z \in S$.
Note that $\kappa_n \to \infty$ implies $y = z$.
Let $\delta_n = ||y_n - z_n||$, and assume that $\delta_n > 0$ for every $n$.
Observe that as $y_n$ and $z_n$ converge to the same point, we have $\delta_n \to 0$,
and, as $[x,y]$ is continuous in $y \in S$ for every $x \in S$, we have also that
$\kappa_n ||y_n - z_n|| = \kappa_n \delta_n \to 0$.
Let $u_n = \frac{y_n}{\kappa_n \delta_n}$ and $v_n = \frac{z_n}{\kappa_n \delta_n}$.
Then $||u_n - v_n|| = \frac{1}{\kappa_n} \to 0$, and $|[x,u_n]-[x,v_n]| = 1$, and hence,
(3) does not hold.
\end{proof}

\section{Proof of Theorem~\ref{thm:main}}

Assume that $A$ is adjoint abelian.
Let $\mu$ and $\nu$ be two different eigenvalues of $A$ and let $x$ and $y$ be eigenvectors belonging to $\mu$
and $\nu$, respectively.
Then,
\[
\mu [x,y] = [Ax,y] = [x,Ay] = \nu [x,y],
\]
which yields that $x$ is transversal to $y$. Thus, any two eigenspaces,
belonging to distinct eigenvalues, are both transversal and normal, which,
in particular, proves (2) (for isometries, see this observation in \cite{KR70}).
Recall that an Auerbach basis of a Minkowski space is a basis
in which any two distinct vectors are transversal and normal to each other, and that
in every norm there is an Auerbach basis.
Note that the restriction of a norm to a linear subspace is also a norm, and thus, we may choose
Auerbach bases in each eigenspace separately, which, by the previous observation, form an Auerbach basis in
the whole space.
Let $x \in \M$, and observe that $x$ has a unique representation of the form
$x=\sum_{i=1}^k x_i$, where $x_i \in \bar{E}_i$.
To prove Theorem~\ref{thm:main}, we need the following lemma.

\begin{lem}\label{lem:decomposition}
If $z \in \bar{E}_i$ for some value of $i$, then $[z,x] = [z,x_i]$.
\end{lem}

\begin{proof}
Assume that $z \in \bar{E}_i$ for some $i \in \{1,2, \ldots, k\}$.

\emph{Case 1}, $i = 1$.\\
If $\lambda_1 = 0$, then $A$ is the zero operator, and the assertion immediately follows.
Let us assume that $\lambda_1 > 0$.
As $A$ is adjoint abelian, we have that $[A^2z,x]=[Az,Ax]=[z,A^2x]$.
Observe that $A^2 z=\lambda_1^2 z$, and that $A^2x=\sum_{i=1}^k \lambda_i^2 x_i$.
Thus,
\begin{equation}\label{eq:first_step}
[z,x]=\left[ z, \sum_{i=1}^k \left( \frac{\lambda_i^2}{\lambda_1^2} \right)^n x_i \right]
\end{equation}
for every positive integer $n$.
Since $[.,.]$ is continuous in both variables, we obtain that
the limit of the right-hand side of (\ref{eq:first_step}) is $[z, x_1]$, and hence, $[z,x] = [z,x_1]$.

\emph{Case 2}, $i > 1$ and $\lambda_i \neq 0$.\\
We prove by induction on $i$.
Let us assume that $[y,x]=[y,x_j]$ for every $y \in \bar{E}_j$ and $j = 1,2, \ldots, i-1$.

First, set $w = \sum_{j=1}^i x_j$ and $F_i = \spn (\bigcup_{j=1}^i \bar{E}_j)$.
We show that $[z,w] = [z,x_i]$.
Note that $A|_{F_i}$ is invertible, adjoint abelian, and its inverse is also
adjoint abelian.
Let $B_i$ denote the inverse of $A |_{F_i}$, and observe that the absolute values of the eigenvalues of $B_i$ are $\frac{1}{\lambda_j}$ and its eigenspaces are $E_j$ and $E_{-j}$, where $j=1,2,\ldots, i$.
Thus, we have $\frac{1}{\lambda_i^2}[z,w]=[B_i^2z,w]=[B_iw,B_iw]=[z,B_i^2w]$ and
\[
[z,w]=\left[ z, \sum_{j=1}^i \left( \frac{\lambda_i^2}{\lambda_j^2} \right)^n x_j \right]
\]
for every positive integer $n$.
By the continuity of the semi-inner-product, and since the limit of the right-hand side is $[z,x_i]$,
we have the desired equality.

Now we show that $[z,x] = [z,x_i]$.
Similarly like before, we obtain that
\[
[z,x]= \left[ z,\sum_{j=1}^k \left( \frac{\lambda_j^2}{\lambda_i^2} \right)^n x_j \right]
\]
for every positive integer $n$.
Observe that $\lim_{n \to \infty} \left|\left| \sum_{j=i+1}^k \left( \frac{\lambda_j^2}{\lambda_i^2} \right)^n x_j
\right|\right| = 0$,
and that, by the previous paragraph,
$\left[ z,\sum_{j=1}^i \left( \frac{\lambda_j^2}{\lambda_i^2} \right)^n x_j \right] = [z,x_i]$
for every positive integer $n$.
Thus, by Lemma~\ref{lem:uniform}, we have that $[z,x] = [z,x_i]$.

\emph{Case 3}, $i>1$ and $\lambda_i = 0$.\\
Note that in this case $i = k$ and $\bar{E}_k = E_k$.
Let $F=\spn(\bigcup_{i=1}^{k-1}\bar{E}_i)$ and set $x_f = \sum_{j=1}^{k-1} x_j$.
By Cases 1 and 2, we have $[x_f,x_k] = \sum_{j=1}^{k-1} [x_j,x_k] = 0$.
On the other hand,
\[
[x_k,x_f]=\left[ x_k, A^2 \sum_{j=1}^{k-1} \frac{1}{\lambda_j^2} x_j \right] =
\left[ A^2 x_k, \sum_{j=1}^{k-1} \frac{1}{\lambda_j^2} x_j\right] = 0.
\]
Thus, we obtained that $F$ and $E_k$ are both transversal and normal.
Now we let $G=\spn\{ x_f,x_k,z \}$.

\emph{Subcase 3.1}, $\dim G = 2$.\\
Let $e_1=\frac{x_f}{||x_f||}$ and $e_2=\frac{x_k}{||x_k||}$.
Since $F$ and $E_k$ are transversal and normal, the pair $\{ e_1, e_2 \}$ is an Auerbach basis in $G$.
By Cases 1 and 2, $[.,.]|_F = \sum_{j=1}^{k-1} [.,.]|_{\bar{E}_j}$, which yields that
$[e_1,\alpha_1 e_1 + \alpha_2 e_2]=\alpha_1$ for any $\alpha_1, \alpha_2 \in \Re$.

Now we identify $G$ with the Euclidean plane $\Re^2$ by $\alpha_1 e_1 + \alpha_2 e_2 \mapsto (\alpha_1,\alpha_2)$;
or in other words, we assume that $e_1$ and $e_2$ are the standard basis of an underlying Euclidean plane.
We need to show that $[e_2,\alpha_1 e_1 + \alpha_2 e_2]=\alpha_2$ for any $\alpha_1,\alpha_2 \in \Re$,
or, equivalently, that the unit circle $S \cap G$ of the subspace $G$ is the Euclidean unit circle.

Since $\M$ is smooth, $S \cap G$ is a convex differentiable curve.
Consider the Descartes coordinate system induced by the standard basis $e_1$ and $e_2$,
and note that the lines $x=1$, $x=-1$, $y=1$ and $y=-1$ support $\conv S$.
Thus, for every value of $x \in (-1,1)$, there is exactly one point of $S$ with $x$ as its $x$-coordinate
and nonnegative $y$-coordinate.
We represent the points of $S \cap G$ with nonnegative $y$-coordinates as the union of the graph of a function
$x \mapsto f(x)$ with $x \in [-1,1]$, and (possibly) two segments on the lines with equations $x=1$ and $x=-1$.
We express the equality $[e_1,\alpha_1 e_1 + \alpha_2 e_2]=\alpha_1$ with the function $f$.

We may assume that $v = \alpha_1 e_1 + \alpha_2 e_2 \in S \cap G$.
Consider the case that $v=x_0 e_1 + f(x_0) e_2$ for some $x_0 \in (-1,1)$.
Then we have $[e_1,v] = [e_1,x_0e_1+f(x_0)e_2]=x_0$.
Let $v_p$ denote the projection of $e_1$ onto the line $\{ \lambda v : \lambda \in \Re \}$
parallel to the supporting line of $\conv S$ at $v$, and let
$e_p$ denote the projection of $v$ onto the line $\{ \lambda e_1 : \lambda \in \Re \}$
parallel to the supporting line of $\conv S$ at $e_1$ (cf. Figure~\ref{fig:0_2dim}).
Let $v_p = \mu v$ and observe that $e_p = x_0 e_1$.
We note that, by the construction of the semi-inner-product described, for example in \cite{L61},
we have that $[e_1,v] = \mu$ and $[v,e_1] = x_0$.
Hence the triangle with vertices $o, e_1, v$ is similar to the triangle with vertices $o, v_p, e_p$,
with similarity ratio $x_0$.
From this, we obtain that $v_p = x_0 v = x_0^2 e_1 + x_0 f(x_0) e_2$.
As the line, passing through $e_1$ and $v_p$, is parallel to the supporting line of $\conv S$ at $v$,
we have
\[
f'(x_0)= -\frac{x_0 f(x_0)}{1-x_0^2},
\]
which is an ordinary differential equation for $f$ with the initial condition $f(0)=1$.

\begin{figure}[here]
\includegraphics[width=0.85\textwidth]{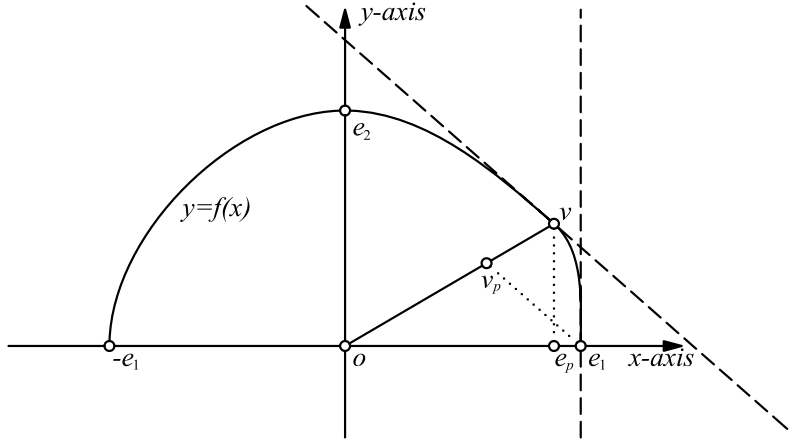}
\caption[]{An illustration for Subcase 3.1}
\label{fig:0_2dim}
\end{figure}

We omit an elementary computation that shows
that the solution of this differential equation is $y=\sqrt{1-x^2}$.
Thus, we obtain that $S \cap G$ is the Euclidean unit circle, which yields, in particular, that
$[z,x_f+x_k] = [z,x_k]$.

\emph{Subcase 3.2}, $\dim G = 3$.\\
Set $e_1=\frac{x_f}{||x_f||}$ and choose an Auerbach basis $\{e_2, e_3 \}$ in $\spn \{ x_k, z \}$.
Then the set $\{ e_1, e_2, e_3 \}$ is an Auerbach basis in $G$.
Furthermore, since $F$ and $E_k$ are transversal and normal, $\{ e_1, v \}$ is an Auerbach basis in its span
for any $v \in \spn \{ x_k, z \}$ with $||v||=1$.
Thus, applying the argument in Subcase 3.1 for the subspace $\spn \{ e_1, v \}$,
we obtain that $S \cap \spn \{ e_1, v \}$ is the ellipse with semiaxes $e_1$ and $v$.
Note that this property and $S \cap \spn \{ e_1, e_2 \}$ determines the norm.

Consider the semi-inner-product defined by $[\beta_1 e_1 + \beta_2 e_2 + \beta_3 e_3, \alpha_1 e_1 + \alpha_2 e_2 + \alpha_3 e_3]' = \beta_1 \alpha_1 + [\beta_2 e_2 + \beta_3 e_3, \alpha_2 e_2 + \alpha_3 e_3]$.
We show that $[.,.]'$ and $[.,.]$ define the same norm, which, as a smooth norm uniquely determines
its semi-inner-product, yields that $[.,.]' = [.,.]$.

Let $v = \alpha_2 e_2 + \alpha_3 e_3$ with $||v|| = 1$ be arbitrary.
Note that if $||\mu e_1 + \nu v||=1$, then
\[
[\mu e_1 + \nu v, \mu e_1 + \nu v]' =  \mu^2 + \nu^2 [v,v] = \mu^2 + \nu^2 =1,
\]
which, in $\spn \{ e_1, v\}$, is the equation of the ellipse with semiaxes $e_1$ and $v$.
As the restrictions of $[.,.]'$ and $[.,.]$ to $\spn \{ e_2, e_3 \}$ are clearly equal,
we obtain that $[.,.]=[.,.]'$, which, in particular, implies that 
$[z,x_f + x_k] = [z,x_k]$.
\end{proof}

By Lemma~\ref{lem:decomposition}, we have that (1) of Theorem~\ref{thm:main} holds.
Thus, it remains to show that (3) also holds.
Without loss of generality, let us assume that $k=1$, and that $\lambda_1 = 1$.
Then every $x \in \M$ can be decomposed as $x=x_1+y_1$
with $x_1 \in E_1$ and $y_1 \in E_{-1}$. 
Hence,
\[
[A(x_1+y_1),A(x_1+y_1)]=[x_1+y_1,A^2(x_1+y_1)]=[x_1+y_1,x_1+y_1],
\]
and thus, $A$ is an isometry.

Finally, we show that if (1), (2) and (3) holds, then $A$ is adjoint abelian.
Let $x=\sum_{i=1}^k x_i$ and $y=\sum_{i_1}^k y_i$ with $x_i,y_i \in \bar{E}_i$.
Assume, first, that $\lambda_k \neq 0$, which means that $A$ is invertible.
Note that $\bar{E}_i$ is an invariant subspace of $A$ for every value of $i$, and
that $\left( \frac{1}{\lambda_i}A \right)^{-1} = \frac{1}{\lambda_i}A = \id$ on $\bar{E}_i$.
Hence,
\[
[Ax,y]=\sum_{i=1}^k [Ax_i,y_i]=\sum_{i=1}^k \lambda_i \left[\frac{1}{\lambda_i}A x_i,y_i\right]=
\sum_{i=1}^k \lambda_i \left[x_i, \left( \frac{1}{\lambda_i}A \right)^{-1} y_i\right]=
\]
\[
\sum_{i=1}^k \lambda_i \left[x_i, \frac{1}{\lambda_i}A y_i\right]= \sum_{i=1}^k [x_i, A y_i] = [x,Ay],
\]
and the assertion follows.
If $A$ is not invertible, we may apply a slighly modified argument.

{\bf Acknowledgements.}
The author is indebted to \'A. G. Horv\'ath for proposing this project and for his helpful comments.

\end{document}